\documentclass{amsart}

\usepackage{amsthm,amsmath,amssymb,mathcomp}
\newtheorem{theorem}{Theorem}[section]
\newtheorem{lemma}[theorem]{Lemma}
\newtheorem{proposition}[theorem]{Proposition}
\theoremstyle{definition}
\newtheorem{definition}[theorem]{Definition}

\theoremstyle{remark}

\numberwithin{equation}{section}

\numberwithin{equation}{section}

\def\ca{{\mathcal A}}
\def\cb{{\mathcal B}}
\def\cc{{\mathcal C}}

\def\cf{{\mathcal F}}

\def\bc{{\mathbb C}}

\def\a{\alpha}

\def\d{\delta}

\def\g{\gamma}

\def\l{\lambda}

\def\r{\rho}
\def\s{\sigma}

\begin{document}

 \title[Complete boundedness of the Schur product]{On the complete boundedness of the Schur block product}

\author[Erik Christensen]{Erik Christensen}

\address{Institute for Mathematical Sciences, Universitetsparken 5, 2100 Copenhagen, Denmark}

\email{echris@math.ku.dk}


\date{\today}

\commby{Stephan Ramon Garcia}

\begin{abstract}
We give a  {\em Stinespring representation } of the Schur block product on pairs of square matrices with  entries in a C$^*$-algebra as a completely bounded bilinear operator of the form:
$$ A:=(a_{ij}), B:= (b_{ij}):  A \square B := (a_{ij}b_{ij}) = V^*\l(A)F \l(B) V,$$
such that $V$ is an isometry, $\l$ is a *-representation and $F$ is a self-adjoint unitary. This implies an inequality due to Livshits and 2 more ones, apparently new, on the diagonals of matrices:
\begin{align*}
\|A \square B\| &\leq \|A\|_r \|B\|_c \text{ operator, row and column norm; } \\ - \mathrm{diag}(A^*A) &\leq A^*\square A \leq  \mathrm{diag}(A^*A), \\
\forall \Xi, \Gamma \in \bc^n\otimes H: \, |\langle (A
\square B) \Xi, \Gamma \rangle| & \leq   \|\big(\mathrm{diag}(B^*B)\big)^{1/2}\Xi\| \|\big(\mathrm{diag}(AA^*)\big)^{1/2}\Gamma\|.
\end{align*}
 \end{abstract}

\maketitle

\section{Introduction}

The {\em Hadamard} or {\em Schur} product between a pair of  scalar matrices of the same shape has been studied for  more than 100 years \cite{Sc, HJ},  and it is closely related to basic mathematical  subjects such as matrix theory and  representation theory. The product also has a natural interest for operator theorists  \cite{Be}, operator algebraists \cite{Pi1} and it is also  used in the study of quantum channels, \cite{AS}. The usage of the names Hadamard and Schur in connection with this product has varied in the literature, and nice expositions on the history behind the use of the names are to be found in  Horn \cite{Ho}, section 2 and Horn \&  Johnson \cite{HJ}, section 5.0.  

  In connection with the theory of {\em operator spaces } and completely bounded mappings on operator algebras \cite{ER, Ha,  Pa,  Pi2} it is obvious to ask questions on the generalization of the Schur product to square  matrices over a C$^*$-algebra.
This extension of the classical Hadamard or Schur product already exists in the theory for matrices and linear algebra \cite{ HM2, HMN, Li}, and our present article extends especially results by Horn, Mathias and Nakamura from \cite{HMN} and Livshits \cite{Li}. We will return to this point, when we have established some more notation. 
    
In \cite{CS} we discussed  a bilinear  mapping $\Phi,$ defined on  the product of a pair of C$^*$-algebras $\ca$, $\cb$ and mapping into  a C$^*$-algebra $\cc.$  We defined $\Phi$ to be completely bounded if there exists a positive constant $K$ such that for any natural number $k$ and for the bilinear operator $\Phi^k$ defined on the $k \times k$ matrices over the algebras $\ca, \cb$ denoted as  $M_k(\ca)$ and $M_k(\cb)  $ into $M_k(\cc)$ by
\begin{equation} \label{liftk}
 \forall A \in M_k(\ca) \, \forall B \in M_k(\cb)\, \forall i,j \in \{1, 2, \dots , k\}:  (\Phi^k(A,B))_{ij}\, :=\,\sum_{l=1}^k\Phi(a_{il},b_{lj}),
\end{equation}
we have $\|\Phi^k\| \leq K.$ 
If a bilinear operator $\Phi$ is completely bounded we define its completely bounded norm by $\|\Phi\|_{cb} := \sup\{\|\Phi^k\|\}.$ The main result of the article \cite{CS} is that a bilinear operator like $\Phi$ on a pair of C$^*$-algebras $\ca, \cb$ into a C$^*$-algebra $\cc$  acting on  a Hilbert space $H,$  is completely bounded if and only if there exist Hilbert spaces $K,\,L,$  $^*-$representations $\l$ of $\ca$ on $K$, $\rho$ of $\cb$ on $L$ and bounded operators $X$ in $B(K,H),$ $Y $ in $B(L,K), $ $Z$ in $B(H,L)$ such that 
\begin{equation} \label{Stine}
\forall a \in \ca \, \forall b \in \cb: \quad \Phi(a,b) = X\l(a)Y\rho(b) Z, \text{ and } \|\Phi\|_{cb} = \|X\|\|Y\|\|Z\|.
\end{equation} 
The decomposition of the bilinear operator $\Phi$ given in (\ref{Stine}) is called a {\em Stinespring representation }  of $\Phi$ in recognition of Stinespring's description of completely positive mappings on C$^*$-algebras, \cite{WS}. 

A linear operator between operator spaces is  defined to be completely bounded if all the natural extensions to matrices over the space are bounded by some fixed number. Given a scalar $ n\times  n $ matrix $A = (a_{ij}) \in M_n(\bc),$ it is  known  \cite{Ha, Sm} that the mapping $S_{A}$ on $M_n(\bc)$ which is induced by  Schur multiplication with $A$ on $M_n(\bc)$ is completely bounded and the completely bounded norm $\|S_{A}\|_{cb} $ equals its norm $\|S_A\|.$ 
    The aim of this article is to prove that for any operator algebra $\ca$ the associative product $\square$ on the algebra $M_n(\ca)$ of $n \times n$ matrices over $\ca,$ which usually is called the Schur product and is defined by $$  \forall A, B  \in M_n(\ca): \quad A\square B :=(a_{ij}b_{ij} ),$$  
is completely bounded with completely bounded norm 1, and that it has a natural decomposition as a difference of 2 {\em positive, } a term which will be explained below, bilinear mappings. We do this by providing an explicit and - in our opinion natural -  {\em Stinespring representation}  of $\square$  as a completely bounded  bilinear operator on $M_n(\ca)$ of norm 1. 
 
Here we pose a warning to avoid too much confusion. The operation $\square$ is defined on $M_n(\ca)$ and is  for any natural number $k $ lifted to an associative product $\square^k$ on $M_k(M_n(\ca))$ via the formula (\ref{liftk}), so we prove that $\sup\{\|\square^k\|\}$ is 1 by showing that if the algebra $\ca$ acts on a Hilbert space $H,$ and $M_n(\ca)$ acts on $\bc^n \otimes H,$ in the natural way,  then there exists  an isometry $V$ of $\bc^n \otimes H$ into $\bc^n \otimes H \otimes \bc^n ,$ a self-adjoijnt unitary $F$ on $\bc^n \otimes H \otimes \bc^n $ and a unital *-representation $\l $ of $M_n(B(H))$ on $\bc^n \otimes H \otimes \bc^n$ such that   
\begin{equation} \label{main}
\forall \, A, B \in \, M_n(\ca): \quad A\square B = V^*\l(A)F\l(B)V. \end{equation} 

\noindent
Since the self-adjoint unitary is a diifference of 2 complementary orthogonal projections $F = P - (I-P),$ we get from the equation (\ref{main}) that  
\begin{equation} \label{difference}
\forall \, A, B \in \, M_n(\ca): \quad A \square B  = V^*\l(A)P \l(B) V - V^*\l(A)(I-P)\l(B)V, \end{equation} 

\noindent so the  Schur block product is written in a natural way as a difference of 2 completely bounded bilinear  operators which, in a natural way, may be called {\em positive.  } It turns out that if we take the {\em absolute value, } which we denote $|\square|,$ in the sense that we replace $F$ by $I$ in (\ref{main}) then we get  

\begin{equation} \label{absolute}
\forall \, A, B \in \, M_n(\ca): \quad A |\square| B = V^*\l(A B) V = \mathrm{ diag}(AB) := \underset{1 \leq i,j \leq n}{\sum} e_{ii} \otimes a_{ij}b_{ji}.  \end{equation}

\noindent
From the equations (\ref{difference}) and ( \ref{absolute}) we get immediately the following operator inequality

\begin{equation} \label{SchurIneq}
\forall \, A \in \, M_n(\ca):\quad -\mathrm{diag} (A^* A)\leq A^* \square A \leq  \mathrm{diag}(A^*A),
\end{equation} 

\noindent
and an inequality which is closely related to the classical Cauchy-Schwarz inequality for positive semidefinite bilinear forms,

\begin{align} \label{SchurCaSc}
\notag \forall \Xi, \Gamma \in \bc^n\otimes  H    & \,\forall A, B  \in \, M_n(B(H)):\\  |\langle (A
\square B) \Xi, \Gamma \rangle|^2 \,  & \leq  \, \langle \mathrm{diag}(B^*B)\Xi, \Xi\rangle \langle \mathrm{diag}(AA^*)\Gamma, \Gamma \rangle \\ & = \notag \, \|\big(\mathrm{diag}(B^*B)\big)^{1/2}\Xi\|^2\|\big(\mathrm{diag}(AA^*)\big)^{1/2}\Gamma\|^2.
\end{align}
 
After we completed the first draft of a presentation of this result we realized that the  operation of constructing $\square^k$ was introduced in \cite{HMN} by Horn, Mathias and Nakamura in the case when the Schur product is the classical one on $M_n(\bc),$ or in other words when the C$^*$-algebra equals $\bc.$ The  article  \cite{HMN} is from 1991, and the quoted result on completely bounded bilinear operators is presented in \cite{CS} from 1987. On the other hand  the result of Lemma 3.2 of  the article \cite{HMN} actually is related to our description of the Schur product $\square$  given in Theorem \ref{mainT} as a completely bounded operator. 

The proof we present below  goes back to the classical result for scalar matrices, which tells that the Schur or Hadamard product may be found as a principal submatrix of the Kronecker product, or in modern terms the tensor product, of the two matrices. This result is presented in \cite{HJ}, Lemma 5.1.1. The new twist is that for block matrices the Schur product is no longer commutative.

The reason why we tried to show complete boundedness of the Schur product came from an inequality in \cite{Ch}, where we studied commutators of the form $[f(D),a]$ where $D$ is an unbounded self-adjoint operator and $f$ is an absolutely continuous function with a certain growth condition. 
 The basic tool we used in that study is a  result on the operator norm of the  Schur product between a pair  of operator valued matrices. 
 This result follows easily from the  description of the Schur product we give below, and is presented as a part of our main theorem.  
In the article \cite{Ch} we got the result as a generalization of  Theorem 1.1 point (i) of  \cite{Be}, in which Bennett studies the  scalar Schur  product. Later on, when working on the present article, we realized that Livshits  already in  \cite{Li} from 1994 published the same inequality. Furthermore the inequality may be seen as an extension of results by Horn, Mathias and Nakamura \cite{ HM1, HMN} on  analogies  to the Cauchy-Schwarz inequality.  
 In order to formulate this result we have to introduce the concepts {\em row norm} and {\em column norm } of a matrix of operators. 
The column norm of a matrix with operator entries is simply the supremum of the norms of the columns from the matrix, when considered as operators. The  row norm is defined in the obvious analogous way, and it equals the column norm of the adjoint operator.  The fundamental norm identity for bounded operators on Hilbert spaces states that $\|x^*x\| = \|x\|^2,$ and based on this we can give the following formal definition. 

\begin{definition} \label{norms}
Let $\ca$ be  subalgebra  of a  C$^*$-algebra $\cc$, 
$J$ a set of indices and $A = (a_{ij}), \, \, i, j \, \in \, J, \, \, a_{ij}  \in \ca$ an $\ca$ valued matrix over $J.$  The column norm $\|A\|_c $ and the row norm $\|A\|_{\mathrm{r}} $ are given by the expressions 
\begin{align*} \label{rcnorms} 
\|A\|_{\mathrm{r}} \,  :=& \underset{i \in J}{\sup} \sqrt{\|\underset{j \in J}{\sum } a_{ij} a_{ij}^*\|} \, = \, \sqrt{\|\mathrm{diag}(AA^*)\|}\\
\|A\|_c \,  :=& \underset{j \in J}{\sup} \sqrt{ \|\underset{i \in J}{\sum} a_{ij}^* a_{ij}\| }\, = \, \sqrt{\|\mathrm{diag}(A^*A)\|}.
\end{align*} 
\end{definition} 

\noindent
Livshits' inequality may then be presented as follows: \newline
For any pair of matrices $A =(a_{ij}), B=(b_{ij})$ indexed over $J$ and  with entries from an operator algebra $\ca,$  the operator norm $\|A \square B \|$ of the Schur block product $A \square B =(a_{ij}b_{ij})$ satisfies the inequality
\begin{equation} \label{norm}
 \|A \square B \| \leq \|A\|_{\mathrm{r}}\|B\|_{\mathrm{c}}.
\end{equation}

\section{The explicit Stinespring form of the Schur block product}
 
We  will now present our decomposition of the Schur block product  for matrices over an operator algebra $\ca,$ and we may and will just as well assume that $ \ca$ is a subalgebra of $B(H)$ for some Hilbert space $H$. The set of indices $J,$ with respect to which we will construct square matrices over $\ca$ may be any set and hence it may be infinite. We will use the symbol $M_J(\ca)$ to denote all square matrices over $  \ca,$  indexed by $J,$ and defining bounded operators on $\ell^2(J, H).$ The point of having $M_J(\ca)$ in mind instead of the larger algebra $M_J((B(H))$ is to underline that the Schur block product is an associative product on the algebra $M_J(\ca).$ On the other hand the description of the product we are going to give will be independent of the algebra $\ca,$  so in the rest of the article  we will just consider the Schur block product as a binary operation on matrices over $B(H).$ The result that the Schur product of bounded matrices is bounded, follows directly  from the description of the product we give.

We will first define the notation we are using.  There is a canonical orthonormal  basis say  $\{\a_j \, : \, j \in J\}$ for $\ell^2(J, \bc)$ and corresponding to this basis there is a set of matrix units  $e_{ij} $ in $B(\ell^2(J, \bc))$ such that  we have $e_{ij}\a_k = \d_{jk}\a_i.$   We will then adopt the notation that for a matrix $A = (a_{ij}) $ in $M_J(B(H))$ we will  represent it as an operator  on the Hilbert space $\ell^2(J, \bc) \otimes H \otimes \ell^2(J, \bc) $ in 3 different ways  as limits of strongly convergent bounded nets, such that each element in each of the nets is a finite  sum of elementary tensor products of operators.
The ordered index set for all 3 nets will be denoted $(\cf(J), \subseteq),$ and it consists of all finite subsets of $J$ ordered by inclusion. 
\begin{align}  
\l(A) &:= \underset{F \in \cf(J)}{ \mathrm{strong \,  limit}} \underset{i,j, k \in F}{\sum} e_{ij}\otimes a_{ij}\otimes e_{kk}, \label{lambda}
\\
\s(A) &:= \underset{F \in \cf(J)}{ \mathrm{strong \,  limit}} \underset{i,j \in F}{\sum} e_{ij}\otimes a_{ij}\otimes e_{ij}, \label{defsigma}
\\
\r(A) &:= \underset{F \in \cf(J)}{ \mathrm{strong \, limit}} \underset{i, k,l\in F}{\sum} e_{ii}\otimes a_{kl}\otimes e_{kl}. \label{rho}
\end{align} 
We assume that it is well known that the mappings $\l, \r,\s$ are faithful *-repre-sentations, i.e. injective self-adjoint  homomorphisms of norm 1. In particular this means that for 2 bounded matrices $A = (a_{ij}) $ and $B = (b_{ij})$ and for a fixed pair $(i,m)$ of indices the infinite sum $\sum_{j \in
J} a_{ij}b_{jm} $ is strongly convergent and defines a matrix element $c_{im}$ of a matrix $C$ in $M_J(B(H)),$ such that $\l(C)$ is given as the strong limit  
\begin{align}
 \l(C) = &\underset{F \in \cf(J)}{ \mathrm{strong \,  limit}} \underset{i,j,k,l,m,n \in F}{\sum} (e_{ij}\otimes a_{ij} \otimes e_{kk})(e_{lm}\otimes b_{lm}\otimes e_{nn}) \\
=& \,\, \, \, \underset{F \in \cf(J)}{ \mathrm{strong \,  limit}} \underset{i,j,k,m \in F}{\sum} e_{im}\otimes a_{ij}b_{jm}  \otimes e_{kk}.  \end{align}

The representations $\l$ and $\r$ are unital, but $ \s$ is not, unless $J$ consists of one element. For $\s$ we get $\s(I) = \sum_{j\in J} e_{jj}\otimes I_{B(H)} \otimes e_{jj}$ which is an orthogonal projection, say $Q,$ from $\ell^2(J, \bc) \otimes H \otimes \ell^2(J, \bc))$ onto the closed subspace $K$ which is spanned by all the vectors of the form $\{\a_j\otimes \xi\otimes \a_j\, : \, j \in J , \, \xi \in H\}.$
The reason why we have attached the names $\l, \r $ and $\s$ to these representations, is that in the case when $ H = \bc$ and $J=\{1, , \dots , n\},$ then $\l$ and $\r$ are named the   left and the right standard representation of $M_n(\bc)$ in the theory of von Neumann algebras and these 2 representations have something in common with the left and the right regular representation of a discrete group. The representation $\s$ is a kind of symmetric mix  and it fits nicely into the description of the Schur block product. Before we can see that, we need a generalization of the Kronecker product to the setting of matrices with operator entries. 

\begin{definition} Let $A =(a_{ij}) $ and $B=(b_{kl})$ be elements in $M_J((B(H)).$ The Kronecker block product of $A$ and $B$ is the  matrix $A*_{KB}B $ in $M_J(M_J(B(H))),$ which is defined by the equation 
 \begin{equation} \label{KB}
A*_{KB}B \, := \, \l(A)\r(B) \, = \, \underset{F \in \cf(J)}{ \mathrm{strong \,  limit}} \underset{i,j,k,l \in F}{\sum} e_{ij}\otimes a_{ij}b_{kl} \otimes e_{kl}.
\end{equation}
\end{definition}
 We may now benefit  from the classical result \cite{HJ} Lemma 5.1.1, which describes the Schur product of two scalar matrices  as a principal submatrix of their Kronecker product. In the setting of (\ref{KB}), the matrix, we are looking at, is of the form  $(J \times J)\times (J \times J)$ and the rows are indexed by pairs $(i,k)$ whereas the columns are indexed by pairs $(j, l)$ and the principal submatrix, which gives the Schur block product, is the one where the index set consists of all the pairs $\{ (j,j)\, : \, j \in J\}.$  Moreover we find right away that the orthogonal projection $Q$ we defined above is exactly the one which supports the principal sub-matrices  which have non-zero entries only on elements which have indices of the form $((i,i), (j,j)).$   Based on this we state without any further proof the following proposition:
 
\begin{proposition}
Let $A =(a_{ij})$ and  $B = (b_{ij} )$ be elements in $M_J((B(H)), $ then their Schur block product $A\square B = (a_{ij}b_{ij}) $ is in $M_J((B(H))$ and \begin{equation} \label{sASB} 
\s(A\square B)  \, = \,\underset{F \in \cf(J)}{ \mathrm{strong \,  limit}} \underset{ij \in F}{\sum} e_{ij} \otimes a_{ij}b_{ij} \otimes e_{ij} \, = \, Q \l(A)\r(B) Q   \, = \,Q(A*_{KB}B)Q .
\end{equation}
 \end{proposition}

It should be remarked, that you may  right away see,  that in the case when $H=\bc,$ then the matrices $A=(a_{ij}) $ and $B=(b_{kl})$ are scalar matrices, and the Kronecker block product is just the well known Kronecker product.  

  The space $K := Q \big(\ell^2(J,\bc) \otimes H \otimes \ell^2(J, \bc)\big)$ is closely related to $\ell^2(J, \bc) \otimes H$ and we  define an isometry $V$ of $\ell^2(J, \bc) \otimes H$ onto  $K$ by 
\begin{equation} \label{V}
 \forall \,  \Xi \in \ell^2(J,\bc) \otimes H,\,\,  \Xi = \sum_{j \in J} \a_j \otimes \xi_j  : \quad V\Xi \, := \, \sum_{j\in J} \a_j\otimes \xi_j\otimes \a_j.
\end{equation} 

\noindent
It is now a matter of computation to verify the following equation, which shows that the representation $\s$ on $K$ is unitarily equivalent to the identity representation of $M_J(B(H))$ on $\ell^2(J, \bc)\otimes H.$

\begin{equation}\label{sigma} 
\forall \, A = (a_{ij}) \in M_J(B(H)): \quad \s(A)V= V A, \text{ or } A = V^* \s(A) V. 
\end{equation}

We may then present our first theorem.

\begin{theorem} \label{thmcb}
The Schur block product is completely bounded, with completely bounded norm 1.
\end{theorem}

\begin{proof}
We give a description of the Schur block product in the form described in equation (\ref{Stine}), so let $A= (a_{ij})$ and $B=(b_{ij}) $ be in $M_J(B(H))$ then 
\begin{align} 
\notag A \square B \, & = \, V^*\s(A \square B)V \text{ by } (\ref{sigma}) \\
\notag &=\, V^* (A*_{KB} B)V \text{ by } (\ref{sASB}) \text{ and } VV^*=Q\\
&= V^* \l(A)\rho(B) V \text{ by } (\ref{KB}) \label{bScb},
\end{align} 
and we have obtained the form from (\ref{Stine}) with operators of norm 1.
\newline Since $M_J(B(H))$ has a unit, the completely bounded norm is 1, and the theorem follows.
   \end{proof}
We could leave the result like this, but we think that the bilinear operators, we look at, do  have many things in common with sesquilinear forms, and in the latter case we do prefer self-adjoint or even better positive semidefinite forms. A similar kind of aesthetics may apply here, so we want to describe the Schur block product not only as a bilinear completely bounded operator, but rather as a difference of 2 positive  completely bounded bilinear operators. We are not far from this in the equation (\ref{bScb}), but  we need to introduce a well known self-adjoint unitary to get the expression, we think may be the {\em right } one.

\begin{definition} \label{flip}
The flip operator $F$ on $\ell^2(J,\bc) \otimes H\otimes \ell^2(J,\bc) $ is defined as the strong limit   $$ F\,:=\, \underset{K \in \cf(J)}{ \mathrm{strong \,  limit}} \underset{ i,j \in F}{\sum} e_{ij} \otimes I \otimes e_{ji} .$$
\end{definition}

\noindent
We have a couple of simple observations, which we collect in the following lemma. 

\begin{lemma} \label{LemmaF}

\begin{itemize}
\item[]
\item[(i)] $F$ is a self-adjoint unitary 
\item[(ii)] $\forall A = (a_{ij}) \in M_J(B(H)): \quad F\l(A)F = \rho(A)$ 
\item[(iii)] $FV \, = \, V .$
\end{itemize}
\end{lemma} 
 \begin{proof} 
It is well known that $F$ is a self-adjoint unitary, which has the property that for $X,Z $ in $B(\ell^2(J, \bc))$ and $Y $ in $B(H)$ we have $F(X \otimes Y \otimes Z) F = Z\otimes Y \otimes X,$ so the statements (i) and (ii) follow. The statement (iii) follows easily once we remark that the subspace $K = Q (\ell^2(J,\bc) \otimes H \otimes  \ell^2(J,\bc)),$ which is the range space of $V,$  is spanned by vectors of the form $\a_i \otimes \xi \otimes \a_i$ and all such vectors are clearly eigenvectors for $F$ corresponding to the eigenvalue 1. 
\end{proof}

Below we list a property of the  isometry $V,$ which  will show why the operator norm of a Schur product $A \square B$ is related to the row norm of $A$ and the column norm of $B$ as described in (\ref{norm}). 

\begin{lemma} \label{columnnorm}
For any matrix $A=(a_{ij}) $ in $M_J(B(H)) $ we have
 \begin{align*}
\|A\|_c &=   \|\l(A)V\|\\
\|A\|_{\mathrm{r}} & = \|V^*\l(A)\|.
\end{align*} 
\end{lemma}
\begin{proof}
We will prove the column case only, since the row case follows by taking adjoints. First remark, that $Q = VV^*,$ so the identity $\|\l(A)Q\| =  \|\l(A)V\|$ follows because $V$ is an isometry. Then let  us compute the square of the norm $ \l(A)Q$ of using the C$^*$-algebraic norm identity.

\begin{align*}   
&\|\l(A)Q\|^2 = \|Q\l(A^*A) Q\| \\ 
&=  \|\underset{F \in \cf(J)}{ \mathrm{strong \,  limit}} (\underset{ i,j,k,l,s,t \in F}{\sum} (e_{ii} \otimes I \otimes e_{ii})(e_{jl}\otimes a_{kj}^*a_{kl} \otimes e_{ss})(e_{tt}\otimes I \otimes e_{tt})\| \\
& \text{ we see that }  i=j=s = t =l \text{ so } \\
& \leq \underset{F \in \cf(J)}{ \limsup}\| 
\underset{  i, k \in F}{\sum} e_{ii} \otimes a^*_{ki}a_{ki} \otimes e_{ii} \| \\
&= \underset{ i \in J}{\sup} \| \underset{k\in J}{\sum}  a^*_{ki}a_{ki} \| = \|A\|_c^2.
  \end{align*} On the other hand, for each $j \in J$ we have $$ \|\l(A)Q \| \geq \|\l(A) (e_{jj} \otimes I_{B(H)} \otimes e_{jj} ) \| = \sqrt{ \|\sum_{i \in J} a_{ij}^*a_{ij}\|},$$ so $\|\l(A)Q\| \geq \|A\|_c$  and the lemma follows. 
\end{proof}
We give the formal definition of the diagonal of a matrix of operators
\begin{definition} \label{diag} 
For an operator $A = (a_{ij} )  $ in $M_J(B(H))$  we will define the diagonal $\mathrm{diag}(A)$ in $M_J(B(H))$  by 
$$ \mathrm{diag}(A)_{ij} = \begin{cases} 0 \quad \text{ if } i \neq j\\
a_{ii} \quad \text{ if } i = j\end{cases}.$$
\end{definition}
There are some simple observations we will use.
\begin{lemma} \label{diaglem}
For an operator $A$ in $M_J(B(H))$ we have \begin{align*}
\mathrm{diag}(A) &= V^*\l(A)V \\
\|A\|_c  & = \|\mathrm{diag}(A^*A)\|^{1/2} \\
\|A\|_r  & = \|\mathrm{diag}(AA^*)\|^{1/2}.
\end{align*}
\end{lemma} 

Then we can state the main result.

\begin{theorem} \label{mainT}
For any Hilbert space $H$  and any set of indices $J:$  \begin{itemize}
\item[(i)] The Schur block product on $M_J(B(H))$ is given by the formula. 
$$ \forall A,\,  B \, \in M_J(B(H)): \quad A \square B \, = \, V^* \l(A) F \l(B) V.$$
\item[(ii)] This is Livshits's inequality, \cite{Li}
$$ \forall A,\,  B \, \in M_J(B(H)): \quad \| A \square B \| \, \leq  \, \|A\|_r\|B\|_c.$$ 
\item[(iii)] Let $X = (x_{ij}), $ and $Y=(y_{ij})  $ be  matrices indexed by  $J$ with elements  from $B(H),  $ and $C$ a non negative real.  If the matrix $Z$ defined by  $ Z = (z_{ij}) := (x_{ij} y_{ij} ) $ is bounded whenever $Y$ is column bounded and satisfies $\|Z\|_{op} \leq C\|Y\|_c, $ then $X$ is row bounded and satisfies $\|X\|_r \leq C.$ 

\item[(iv)]
$$ \forall A \, \in M_J(B(H)):  -\mathrm{diag}(A^*A) \leq A^* \square A \leq \mathrm{diag}(A^*A).$$
\item[(v)]
\begin{align*}&\forall \Xi, \Gamma \in \ell^2(J,\bc)\otimes H   \,\forall A, B  \in \, M_J(B(H)):\\ 
& |\langle (A
\square B) \Xi, \Gamma \rangle| \leq   \|\big(\mathrm{diag}(B^*B)\big)^{1/2}\Xi\| \|  \big(\mathrm{diag}(AA^*)\big)^{1/2}\Gamma\|.
\end{align*}
 \end{itemize}
 \end{theorem} 
\begin{proof}
For $A, B$ in $M_J(B(H))$ we have 
\begin{align*}
A\square B \, &= \, V^*\l(A)\rho(B)V \text{ by } (\ref{bScb}) \\
&= \, V^*\l(A)F\l(B)F V \text{ by  Lemma } \ref{LemmaF} \text{ (ii)} \\ 
&= \, V^*\l(A)F\l(B) V \text{ by  Lemma } \ref{LemmaF} \text{ (iii)}, \\ 
\end{align*}
and the claim (i) is proved. 

The claim (ii) follows from (i) and the result in Lemma \ref{columnnorm}.

Item (iii) shows that Livshits's inequality determines the row norm and by symmetry the column norm as well. To prove it, let  $ k $ be in $J, $ then we can estimate the norm of the $k$'th row of $X$ via the assumptions made.  We  define the matrix   $Y$ by $y_{ij} = 0$ if $i \neq k, j \in J$ and $y_{kj} = 1 $ for all $j$ in $J. $
Then $Y$ is column bounded with column norm $1$ and the norm of the
 matrix $Z := (x_{ij} y_{ij} ) $
  is exactly the norm of the $k$'th row of $X.$       

The statement in (iv) is a direct consequence of (i), Lemma \ref{diaglem} and the fact that $- \l(A^*A) \leq \l(A^*)F\l(A) \leq \l(A^*A).$

With respect to item (v), we find from the classical Cauchy-Schwarz inequality, the statement in (i) and the  Lemma \ref{diaglem} that 
\begin{align*}
\forall \Xi, \Gamma\in \ell^2(J,\bc)\otimes H   \,\forall A, B & \in \, M_J(B(H)):\\  |\langle (A
\square B) \Xi, \Gamma \rangle|^2 & = |\langle V^*\l(A) F \l(B) V\Xi, \Gamma \rangle |^2 \\
& \leq   \|\l(B)V\Xi\|^2 \|\l(A^*)V \Gamma\|^2  \\
&= \langle V^*\l(B^*B)V\Xi, \Xi\rangle \langle V^*\l(AA^*)V\Gamma, \Gamma \rangle \\
&= \langle  \mathrm{diag}(B^*B)\Xi, \Xi\rangle \langle \mathrm{diag}(AA^*)\Gamma, \Gamma \rangle, \end{align*}
  and by taking square roots 
  \begin{equation*}
\, \, \, \quad \quad  \quad \quad \quad \quad  |\langle (A
\square B) \Xi, \Gamma \rangle| \leq   \|\big(\mathrm{diag}(B^*B)\big)^{1/2}\Xi\| \|\big( \mathrm{diag}(AA^*)\big)^{1/2}\Gamma\|.
\end{equation*}
\end{proof}

\section{An elementary observation}
It is worth to remark, that the statement (iv) in  Theorem \ref{mainT} above implies Livshits' inequality, and that (iv)  is an  easy consequence of the ordinary Cauchy-Schwarz inequality as the few lines of computations below show. Hence the validity of the inequality (iv) may  have been realized by many people before, but may be not linked to the complete boundedness of the Schur product. In fact we find by 2 applications of  Cauchy-Schwarz inequalities for Hilbert spaces and for numbers respectively that for any index set $J,$ vectors $\Xi = ( \xi_j),$ $\Gamma = (\g_j),$ in $\ell^2(J, H)$ and  bounded matrices of operators $A= (a_{ij}) $ and $B= (b_{ij})$ in $M_J(B(H))$ we have
\begin{align*}
 |\langle (A
\square B) \Xi, \Gamma \rangle|^2 & = |\underset{ i,j \in J}{\sum} \langle b_{ij} \xi_j, a_{ij}^* \g_i \rangle |^2 \\
 & \leq \bigg( \underset{ i ,j\in J}{\sum} \| b_{ij} \xi_j\|^2\bigg) \bigg(\underset{ i,j \in J}{\sum} \| a_{ij}^* \g_i\|^2\bigg) \\
 &= \bigg(\underset{ j \in J}{\sum} \langle  \underset{ i \in J}{\sum} b_{ij}^*b_{ij} \xi_j, \xi_j \rangle \bigg) \bigg(\underset{ i \in J}{\sum} \langle \underset{ j \in J}{\sum}  a_{ij}a_{ij}^* \g_i, \g_i \rangle \bigg)\\ 
&= \langle\mathrm{diag}(B^*B)\Xi, \Xi\rangle \langle  \mathrm{diag}(AA^*)\Gamma, \Gamma \rangle.
\end{align*}

\end{document}